\documentclass{amsart} 
\address{Department of Mathematics, Institute of Science Tokyo, 2-12-1 Ookayama, Meguro-ku, Tokyo, 152-8551, Japan}
\email{tanaka.a.2255@m.isct.ac.jp}

\usepackage{amsmath,amssymb}
\usepackage{amsfonts}
\usepackage{graphicx} 

\usepackage{hyperref}
\usepackage{caption} 
\usepackage{mathtools}
\usepackage{enumitem}
\usepackage[margin=10pt]{caption} 

\usepackage[final]{pdfpages}
\usepackage{amsthm}
\usepackage{color}

\usepackage{tikz}  

\theoremstyle{plain}
\newtheorem{thm}{Theorem}[section]

\newtheorem{cor}[thm]{Corollary}

\newtheorem*{thm*}{Theorem}
\newtheorem*{cor*}{Corollary}

\theoremstyle{definition}

\newtheorem*{que*}{Question}
\newtheorem*{con*}{Conjecture}

\begin{document}

\title[Combinatorial extension of a construction of Lefschetz fibrations]{Combinatorial extension of a simple construction of Lefschetz fibrations}
\author{Atsushi Tanaka}
\date{}

\begin{abstract}
In a previous work, we introduced a simple and systematic method for constructing a positive allowable Lefschetz fibration (PALF) from a 2-handlebody decomposition of a given Stein surface. 
In this paper, we present a combinatorial extension of this construction, focusing on the flexibility of the regular fiber. 
By introducing variations in the isotopy of the 0-handle during the construction process, we obtain PALFs whose total spaces are diffeomorphic to the original Stein surface but which possess different regular fibers. 
As a primary application, we prove the existence of PALFs with genus $1$ regular fibers whose total spaces are diffeomorphic to the knot traces of Legendrian positive twist knots and positive torus knots $T_{2, 2n+1}$. 
Furthermore, we explicitly compare our PALF associated with the positive torus knot $T_{2, 2n+1}$ to the specific open book decomposition generated by Avdek's Algorithm 2, demonstrating that the regular fiber and monodromy of our construction coincide with the page and monodromy of the corresponding open book.
\end{abstract}

\maketitle

\section{Introduction}\label{sec:intro}

Loi and Piergallini \cite{MR1835390}, Akbulut and Ozbagci \cite{MR1825664}, and Akbulut and Arikan \cite{MR2972525} independently showed that every compact Stein surface (a Stein surface in short) admits a positive allowable Lefschetz fibration over the disk $D^2$ with bounded fibers (PALF in short), and they provided explicit constructions of such fibrations. 

In our previous work \cite{Tanaka2025construction}, we presented a simple and systematic method for constructing a PALF from a 2-handlebody decomposition of any given Stein surface. 
This method produces PALFs whose regular fibers have a small genus and provides an alternative constructive proof of the above result. 
In that paper, to ensure that all vertical segments lifted over 1-handles have northwest (NW) corners as their upper endpoints, we performed an SW stabilization at each NE corner. 
This specific approach allowed the regular fiber of the constructed PALF to be explicitly presented as a plumbing of Hopf bands.

In general, for a constructed PALF, the shape of the regular fiber can be flexibly deformed---while preserving the PALF structure---by an isotopy of the 0-handle, or by sliding the attaching spheres of the 1-handles along the boundary of the union of the 0- and 1-handles. 
Motivated by this flexibility, in this paper, we explore presentations of regular fibers that are well-suited for such deformations. 

By introducing variations in the isotopy of the 0-handle during the construction process, we provide a method to construct PALFs from a given knot in grid position such that they share the same diffeomorphism type, but possess different regular fibers (e.g., with different genera). 
As a corollary to Theorem 3.2 in \cite{Tanaka2025construction}, we show that this modified construction indeed yields a PALF diffeomorphic to the given Stein surface. 

As an application of this construction, we prove the existence of PALFs whose regular fibers have genus $1$ and whose total spaces are diffeomorphic to the knot traces of positive twist knots and positive torus knots $T_{2, 2n+1}$.

Furthermore, it is well known that the contact structure induced on the boundary of a Stein surface is supported by the open book decomposition induced by the PALF that the Stein surface admits (cf.\ \cite[Subsection~2.1]{MR3609904}). Comparing the open book induced by our method with those obtained via existing algorithms clarifies the geometric characteristics of our construction and provides insights for future research. In Section~\ref{sec:open_book}, as a first step, we explicitly compare our PALF associated with the positive torus knot $T_{2, 2n+1}$ to the specific open book generated by Avdek's Algorithm 2 \cite{MR3071137}. By applying geometric deformations to the regular fiber and elementary transformations to the monodromy, we verify that the two constructions exactly coincide.

This paper is organized as follows. 
In Section~\ref{sec:combinatorial}, we introduce the combinatorial extension of our construction method for 2-handlebodies both with and without 1-handles, demonstrating how variations in the isotopy of the 0-handle affect the regular fiber. In addition, we provide practical remarks on the construction process and establish a general principle regarding the vertical translation of a knot in grid position. 
In Section~\ref{sec:application}, we apply this method to prove the existence of genus $1$ regular fibers for specific knot traces, including positive twist knots and positive torus knots. 
 In Section~\ref{sec:open_book}, we compare our PALF associated with the positive torus knot $T_{2, 2n+1}$ to the specific open book generated by Avdek's Algorithm 2. 
Finally, in Section~\ref{sec:variation}, we explore further variations of our construction method such as altering the direction of operations or passing 1-handles behind the 0-handle.

\section*{Acknowledgements}
The author would like to express his sincere gratitude to Professor Hisaaki Endo for his invaluable guidance and constant encouragement to expand the scope of this research into adjacent rich areas such as 3-manifolds. The author is also grateful to Nobuo Iida for his helpful discussions during the early stages of this research. This work was supported by JST SPRING, Japan Grant Number JPMJSP2180.

\section{Modified construction of PALFs via 0-handle isotopies} \label{sec:combinatorial}

In this section, we refer to Sections 2 and 3 of \cite{Tanaka2025construction}.

\subsection{2-handlebodies without 1-handles}\label{subsec:construction1}

As a concrete example, as in \cite{Tanaka2025construction}, we consider the Stein surface $\Pi$ whose handle decomposition consists of a 0-handle and a single 2-handle attached along the Legendrian right-handed trefoil knot $T_{2,3}$ in $(S^3, \xi_{st})$, denoted by $\widetilde{C_0}$ (Figure~5(a) in \cite{Tanaka2025construction}). 

Figures~\ref{fig:X0201}(a)--(d) are reproduced from Figures~5(b), 6(b), 9(a), and 9(d) in \cite{Tanaka2025construction}. The knot $\widetilde{C_0'}$ in grid position is obtained from the front projection of the Legendrian knot $\widetilde{C_0}$, which serves as the attaching circle of the original Stein surface $\Pi$. To ensure that all vertical segments lifted over 1-handles have northwest (NW) corners as their upper endpoints, we perform an SW stabilization at each NE corner. By applying this operation to $\widetilde{C_0'}$, we obtain the knot $\widetilde{C_0''}$ in grid position. The guide line $B_0$ is a copy of $\widetilde{C_0''}$ drawn on the 0-handle $D^2$, which is represented as a gray square. By applying our construction method, we obtain the PALF $P$ shown in Figure~\ref{fig:X0201}(d).

The regular fiber of the PALF $P$ can be deformed into the one shown in Figure~\ref{fig:X0203} through the following isotopies. In Figure~\ref{fig:X0201}(d), Columns~1 and 2 correspond to vertical segments that originally had NW corners in Figure~\ref{fig:X0201}(a). These are deformed by sliding the attaching spheres of the 1-handles along the boundary of the 0-handle, as illustrated in Figure~\ref{fig:X0202}(a). Similarly, Columns~3 and 4 in Figure~\ref{fig:X0201}(d), which originally corresponded to vertical segments with NE corners in Figure~\ref{fig:X0201}(a), are deformed as shown in Figure~\ref{fig:X0202}(b). Consequently, the size of the grid diagram---which had increased due to the SW stabilizations at the NE corners---is restored to its original size by removing one row and one column at each of the two NE corners. Throughout this paper, we adopt the representation of 1-handles shown in Figure~\ref{fig:X0203}.

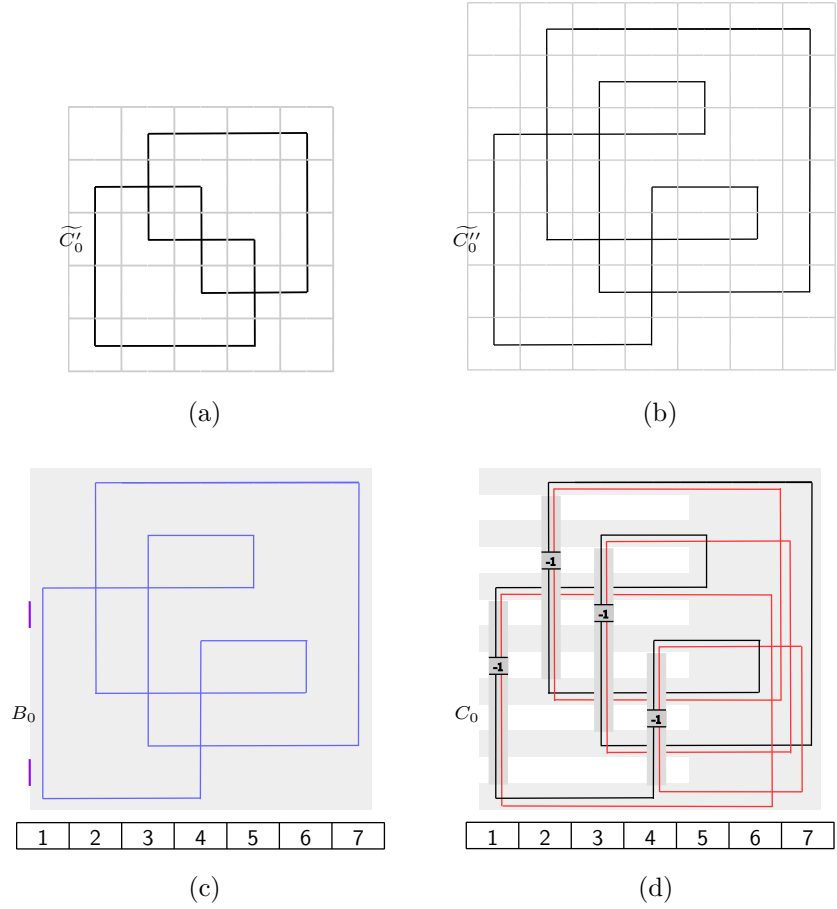
\begin{figure}[htbp]
\centering  
\begin{tikzpicture}
    \node[anchor=south west, inner sep=0] (image) at (0,0)  {\includegraphics[scale=0.7]{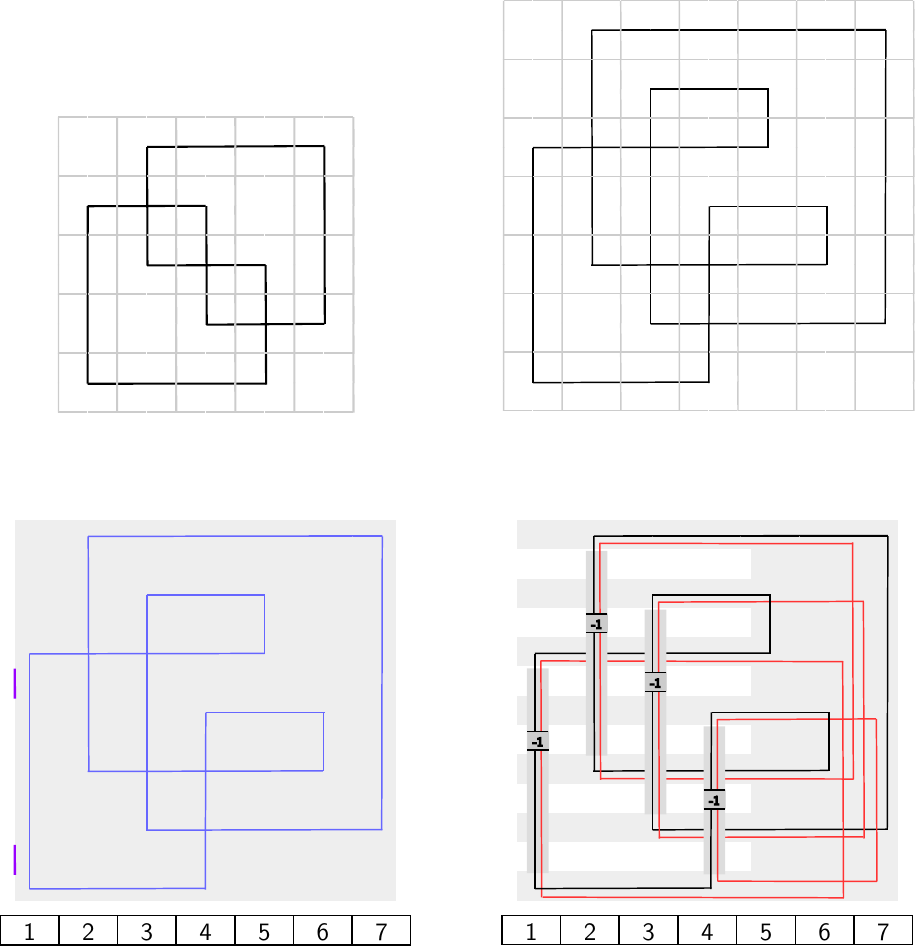}};
    \begin{scope}[x={(image.south east)},y={(image.north west)}]
        \node [font=\scriptsize] at (0.068, 0.72) {$\widetilde{C_0'}$};
        \node [font=\scriptsize] at (0.55, 0.72) {$\widetilde{C_0''}$};
        \node [font=\scriptsize] at (0.01, 0.16) {$B_0$};
        \node [font=\scriptsize] at (0.55, 0.16) {$C_0$};
        \node at (0.23, 0.513) {(a)};
        \node at (0.785, 0.513) {(b)};
        \node [below] at (0.23, -0.02) {(c)};
        \node [below] at (0.78, -0.02) {(d)};
    \end{scope}
\end{tikzpicture}
\caption{(a) The knot $\widetilde{C_0'}$ in grid position. (b) The knot $\widetilde{C_0''}$ in grid position obtained by performing an SW stabilization at an NE corner and a horizontal commutation. (c) The guide line $B_0$ drawn on the 0-handle $D^2$. (d) The PALF $P$ with the monodromy factorization $(C_0, C_4, C_3, C_2, C_1)$. $C_i$ ($1 \leq i \leq 4$) denotes a red simple closed curve passing over the 1-handle in the $i$-th column.}
\label{fig:X0201}
\end{figure}

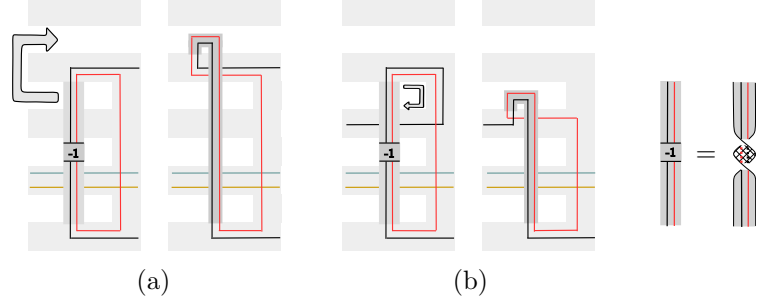
\begin{figure}[htbp]
\centering  
\begin{tikzpicture}
    \node[anchor=south west, inner sep=0] (image) at (0,0)  {\includegraphics[scale=0.75]{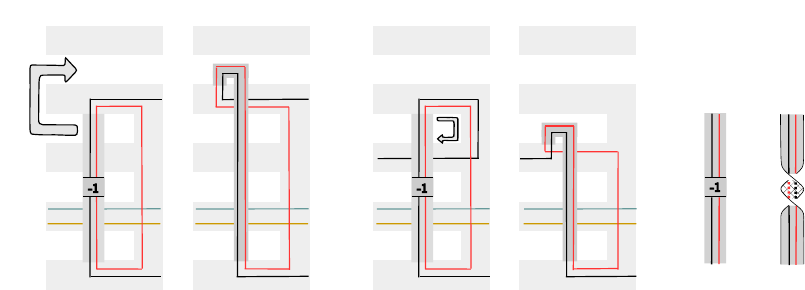}};
    \begin{scope}[x={(image.south east)},y={(image.north west)}]
        \node [below] at (0.22, -0.03) {(a)};
        \node [below] at (0.63, -0.03) {(b)};
         \node at (0.935, 0.34) {$=$};
    \end{scope}
\end{tikzpicture}
\caption{Sliding the attaching sphere of a 1-handle on the boundary of the 0-handle.}
\label{fig:X0202}
\end{figure}

\begin{figure}[htbp]
\centering  
\begin{tikzpicture}
    \node[anchor=south west, inner sep=0] (image) at (0,0)  {\includegraphics[scale=1]{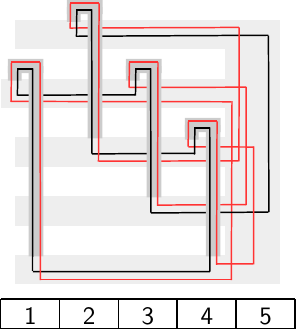}};
    \begin{scope}[x={(image.south east)},y={(image.north west)}]
        \node [font=\scriptsize] at (0.055, 0.62) {$C_0$};
    \end{scope}
\end{tikzpicture}
\caption{The resulting PALF obtained by applying the sliding operations of 1-handles to Columns~1 through 4 of Figure~\ref{fig:X0201}(d), with the monodromy factorization $(C_0, C_4, C_3, C_2, C_1)$. $C_i$ ($1 \leq i \leq 4$) denotes a red simple closed curve passing over the 1-handle in the $i$-th column.}
\label{fig:X0203}
\end{figure}

In \cite{Tanaka2025construction}, during the construction of the regular fiber, the 0-handle was deformed into a comb-like shape by an isotopy parallel to each row. In the present paper, we allow for more flexible deformations by employing isotopies in the vertical direction, provided that they are not obstructed by the horizontal segments of the guide line $B_0$.

A concrete example of the construction of a PALF is shown in Figure~\ref{fig:X0204}. 
We draw a copy of the knot $\widetilde{C_0'}$ in grid position, shown in Figure~\ref{fig:X0201}(a), on the 0-handle $D^2$ (which is represented as a gray square) to serve as the guide line $B_0$ (see Figure~\ref{fig:X0204}(a)). 
The first column of the grid diagram contains a vertical segment of the guide line $B_0$, whose two endpoints (marked in purple) lie on the boundary of the 0-handle. We then attach a 1-handle at this position.
When a vertical segment of the guide line $B_0$ has an NW corner (resp.\ an NE corner), we attach a 1-handle to the boundary of the 0-handle to lift this segment, as shown in Figure~\ref{fig:X0205}(a) (resp.\ (b)).

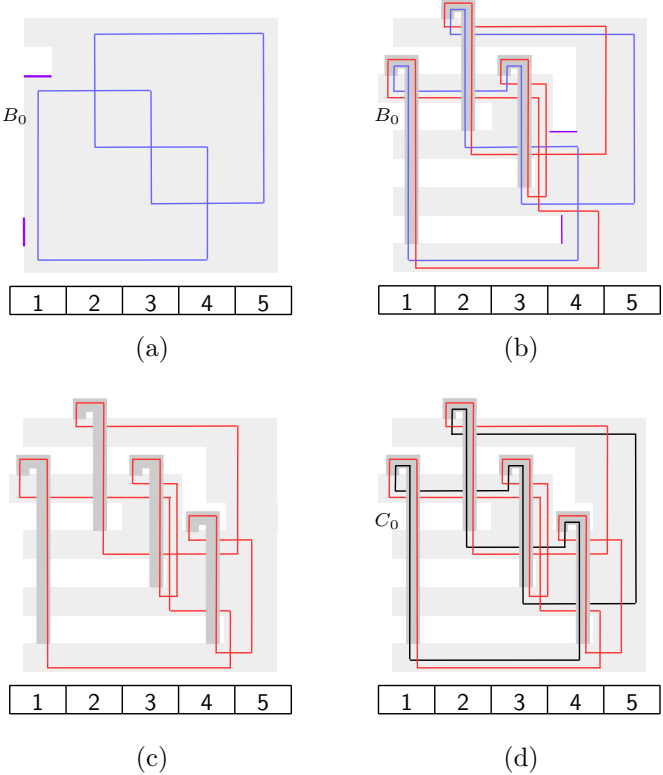
\begin{figure}[htbp]
\centering  
\begin{tikzpicture}
    \node[anchor=south west, inner sep=0] (image) at (0,0)  {\includegraphics[scale=0.75]{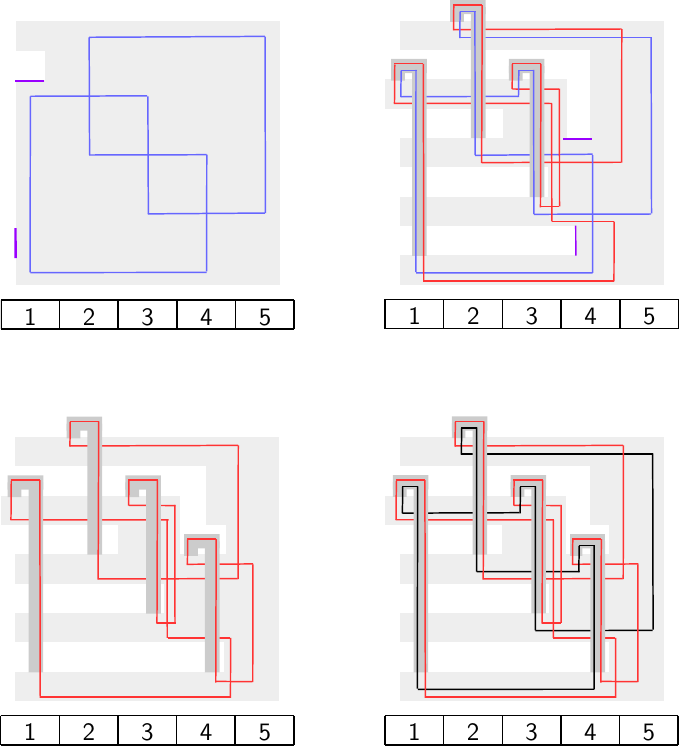}};
    \begin{scope}[x={(image.south east)},y={(image.north west)}]
        \node [font=\scriptsize] at (0.01, 0.835) {$B_0$};
        \node [font=\scriptsize] at (0.58, 0.835) {$B_0$};
        \node [font=\scriptsize] at (0.58, 0.27) {$C_0$};
        \node at (0.22, 0.51) {(a)};
        \node at (0.785, 0.51) {(b)};
        \node [below] at (0.22, -0.03) {(c)};
        \node [below] at (0.785, -0.03) {(d)};
    \end{scope}
\end{tikzpicture}
\caption{A concrete example of the modified construction of a PALF. (a) The guide line $B_0$ drawn on the 0-handle $D^2$. (b) Deforming the boundary of the 0-handle for attaching a 1-handle. (c) The PALF $SF$ with the monodromy factorization $(C_4, C_3, C_2, C_1)$. (d) The PALF $P$ with the monodromy factorization $(C_0, C_4, C_3, C_2, C_1)$. $C_i$ ($1 \leq i \leq 4$) denotes a red simple closed curve passing over the 1-handle in the $i$-th column.}
\label{fig:X0204}
\end{figure}

\begin{figure}[htbp]
\centering  
\begin{tikzpicture}
    \node[anchor=south west, inner sep=0] (image) at (0,0)  {\includegraphics[scale=0.9]{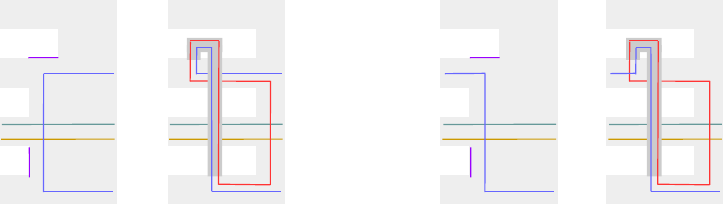}};
    \begin{scope}[x={(image.south east)},y={(image.north west)}]
        \node [below] at (0.2, -0.04) {(a)};
        \node [below] at (0.81, -0.04) {(b)};
    \end{scope}
\end{tikzpicture}
\caption{Attaching a 1-handle to the boundary of the 0-handle to lift a vertical segment of the guide line with (a) an NW corner and (b) an NE corner. The red curve represents the newly added simple closed curve.}
\label{fig:X0205}
\end{figure}

In Columns~2 and 3, we deform the boundary of the 0-handle into a comb-like shape by an isotopy parallel to each row, position both the upper and lower-left endpoints of the vertical segments of the guide line $B_0$ on the boundary of the 0-handle, and attach 1-handles. 
In Column~4, to position the upper endpoint of the vertical segment of the guide line $B_0$ on the boundary of the 0-handle, we push down the boundary of the 0-handle from above and attach a 1-handle (see Figure~\ref{fig:X0204}(b)). 
Then, by removing the guide line $B_0$, we construct the PALF $SF$ whose total space is diffeomorphic to $D^4$ (see Figure~\ref{fig:X0204}(c)).
Furthermore, by placing the closed curve $C_0$ exactly where the guide line $B_0$ was removed, we obtain the PALF $P$ (see Figure~\ref{fig:X0204}(d)). 
We assign the monodromy factorization $(C_4, C_3, C_2, C_1)$ to the PALF $SF$, and the monodromy factorization $(C_0, C_4, C_3, C_2, C_1)$ to the PALF $P$. 
We can confirm that $SF$ and $P$ are PALFs by checking the conditions given in Subsection~2.3 of \cite{Tanaka2025construction}.

We show that the PALF $P$ is diffeomorphic to the original Stein surface $\Pi$. 
A Kirby diagram $KD'$ of the original Stein surface $\Pi$ is shown in Figure~\ref{fig:X0206}(a). 
The curve $\overline{C'_0}$ is the attaching circle of the 2-handle corresponding to $\widetilde{C'_0}$, which represents a Legendrian knot with a writhe of $3$ and two left cusps. 
By the theorem of Eliashberg and Gompf (cf.\ Theorem~2.1 in \cite{Tanaka2025construction}), the framing of $\overline{C'_0}$ must be $0$ (since $tb(\widetilde{C'_0}) - 1 = (3 - 2) - 1 = 0$). 

Next, we construct a Kirby diagram $KD$ representing a 4-dimensional 2-handlebody diffeomorphic to the total space of the given PALF $P$, using the procedure $\Phi$ described in Subsection~2.3 of \cite{Tanaka2025construction} (see Figure~\ref{fig:X0206}(b)). 
In the Kirby diagram $KD$, $\overline{C_i}$ is the attaching circle of the 2-handle corresponding to the vanishing cycle $C_i$ for each $0 \leq i \leq 4$.
The framing of $\overline{C_i}$ is set to be one less than its surface framing induced by the regular fiber.
Since the surface framing of $\overline{C_0}$ is $1$, its framing is $0$. 
Since the surface framing of each $\overline{C_i}$ ($1 \leq i \leq 4$) is $-1$, its framing is $-2$.

\begin{figure}[htbp]
\centering  
\begin{tikzpicture}
    \node[anchor=south west, inner sep=0] (image) at (0,0)  {\includegraphics[scale=0.9]{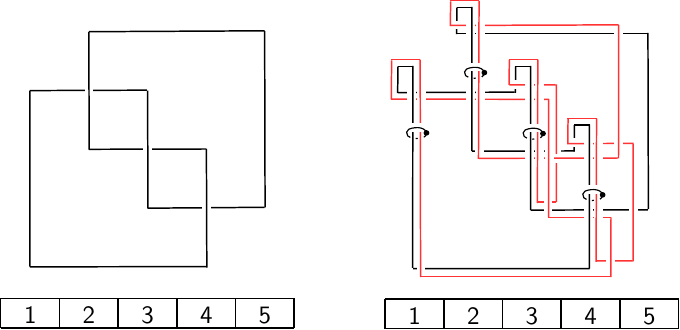}};
    \begin{scope}[x={(image.south east)},y={(image.north west)}]
        \node [font=\scriptsize] at (0.011, 0.65) {$\overline{C_0'}$};
        \node [font=\scriptsize] at (0.578, 0.65) {$\overline{C_0}$};
        \node [font=\scriptsize] at (0.645, 0.65) {$\overline{C_1}$};
        \node [font=\scriptsize] at (0.735, 0.97) {$\overline{C_2}$};
        \node [font=\scriptsize] at (0.82, 0.775) {$\overline{C_3}$};
        \node [font=\scriptsize] at (0.86, 0.68) {$\overline{C_4}$};
        \node [below] at (0.22, -0.04) {(a)};
        \node [below] at (0.787, -0.04) {(b)};
    \end{scope}
\end{tikzpicture}
\caption{(a) A Kirby diagram $KD'$ of the original Stein surface $\Pi$. The curve $\overline{C'_0}$ is the attaching circle of the 2-handle with framing $0$. (b) A Kirby diagram $KD$ of the total space of the PALF $P$. The framing of $\overline{C_0}$ is $0$, and the framing of each $\overline{C_i}$ ($1 \leq i \leq 4$) is $-2$.}
\label{fig:X0206}
\end{figure}
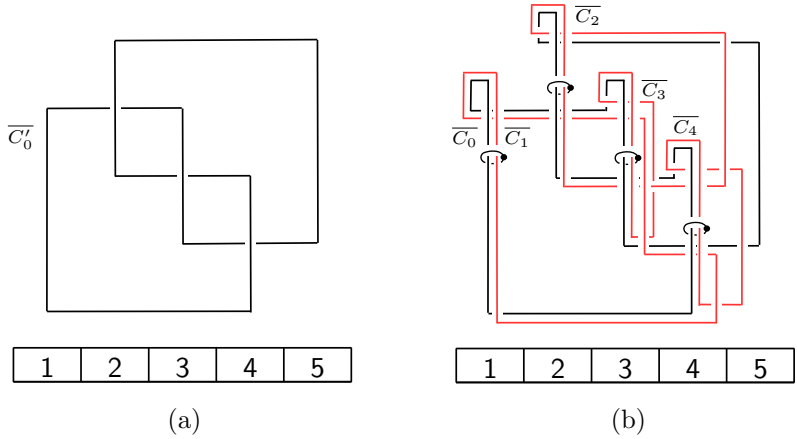

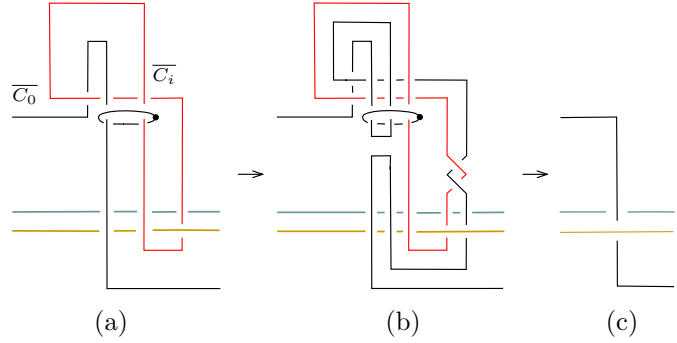
\begin{figure}[htbp]
\centering  
\begin{tikzpicture}
    \node[anchor=south west, inner sep=0] (image) at (0,0)  {\includegraphics[scale=0.5]{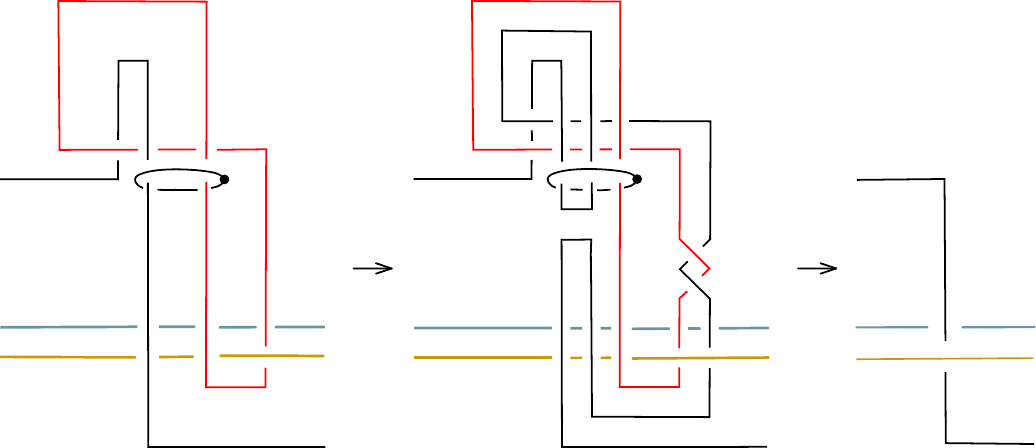}};
    \begin{scope}[x={(image.south east)},y={(image.north west)}]
        \node [font=\scriptsize] at (0.02, 0.68) {$\overline{C_0}$};
        \node [font=\scriptsize] at (0.23, 0.75) {$\overline{C_i}$};
        \node [below] at (0.15, -0.04) {(a)};
        \node [below] at (0.59, -0.04) {(b)};
        \node [below] at (0.92, -0.04) {(c)};
    \end{scope}
\end{tikzpicture}
\caption{A Kirby move applied to the Kirby diagram $KD$. The framing of each $\overline{C_i}$ ($1 \leq i \leq 4$) is $-2$, and the framing of $\overline{C_0}$ remains unchanged.}
\label{fig:X0207}
\end{figure}

We illustrate the sequence of Kirby moves applied to the Kirby diagram $KD$ to transform it into $KD'$, proceeding sequentially from Column~4 down to Column~1 (see Figure~\ref{fig:X0207}). 
In this figure, the red curve represents $\overline{C_i}$ passing over the 1-handle in the $i$-th column, and the black curve represents $\overline{C_0}$. 
Let the orange segments represent the portions of the attaching circles $\overline{C_j}$ ($1 \leq j < i$) that pass over the right part of $\overline{C_i}$. 
Similarly, let the dark green segments represent the portions of the attaching circle $\overline{C_0}$ itself that pass under the depicted segment of $\overline{C_0}$. 
We slide the 2-handle $\overline{C_0}$ over $\overline{C_i}$, and subsequently cancel the 1-handle/2-handle pair consisting of the dotted circle and $\overline{C_i}$. 
Through this sequence of moves from (a) to (b) and finally to (c), the framing of $\overline{C_0}$ remains invariant. 
The orange segments now pass over $\overline{C_0}$, and the dark green segments pass under $\overline{C_0}$. 

In Figure~\ref{fig:X0206}(b), the framing of $\overline{C_0}$ is $0$. 
The vertical stacking order of the attaching circles from bottom to top is $\overline{C_0}, \overline{C_4}, \ldots, \overline{C_1}$. 
Performing the Kirby move shown in Figure~\ref{fig:X0207} at Column~4 leaves the framing of $\overline{C_0}$ invariant at $0$, and changes the vertical stacking order of the attaching circles from bottom to top to $\overline{C_0}, \overline{C_3}, \ldots, \overline{C_1}$.
In this manner, by applying the Kirby moves sequentially from Column~4 to Column~1, we arrive at the Kirby diagram $KD'$.
Therefore, the PALF $P$ is diffeomorphic to the original Stein surface $\Pi$.

From the above discussion, we obtain the following as a corollary to Lemma~3.1 in \cite{Tanaka2025construction}.

\begin{cor}\label{cor:construction1}
The total space of the PALF corresponding to the ordered collection of simple closed curves on the surface with boundary, obtained by the modified construction above, is diffeomorphic to the 4-dimensional handlebody consisting of a 0-handle and a single 2-handle attached along $\widetilde{C_0}$ with framing $tb(\widetilde{C_0})-1$.
\end{cor}

\subsection{2-handlebodies with 1-handles}\label{subsec:construction2}

As a concrete example, following \cite{Tanaka2025construction}, we consider the original Stein surface $\Pi$ consisting of a 0-handle, one 1-handle, and two 2-handles. 
These 2-handles are attached along the components $\widetilde{C_{01}}$ and $\widetilde{C_{02}}$ of a Legendrian link with framings $-2$ and $0$, respectively.
Figure~\ref{fig:X0208}(a) is reproduced from Figure~15(b) in \cite{Tanaka2025construction}. 
The link components $\widetilde{C_{01}'}$ and $\widetilde{C_{02}'}$ in grid position are obtained from the front projection of the Legendrian link consisting of $\widetilde{C_{01}}$ and $\widetilde{C_{02}}$ (see Figure~15(a) in \cite{Tanaka2025construction}). A small square represents the hole of $S^1 \times D^1$.

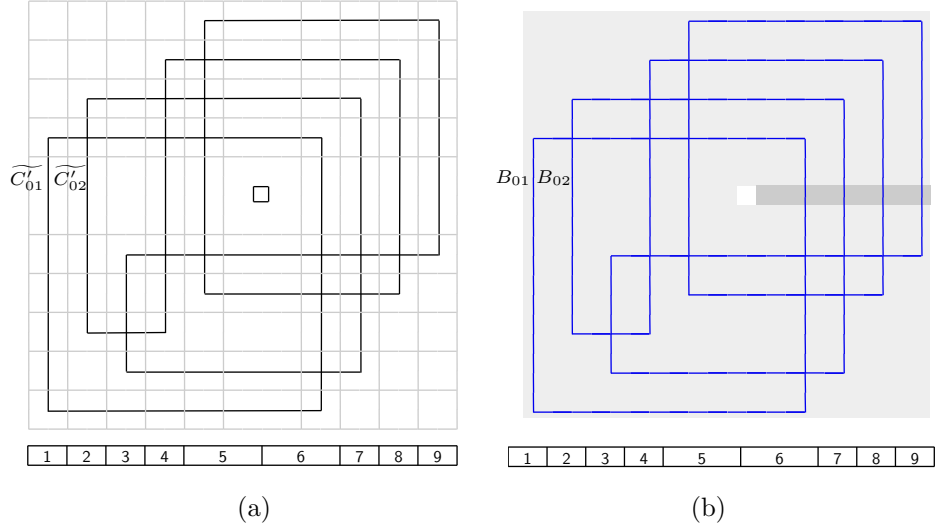
\begin{figure}[htbp]
\centering  
\begin{tikzpicture}
    \node[anchor=south west, inner sep=0] (image) at (0,0)  {\includegraphics[scale=0.7]{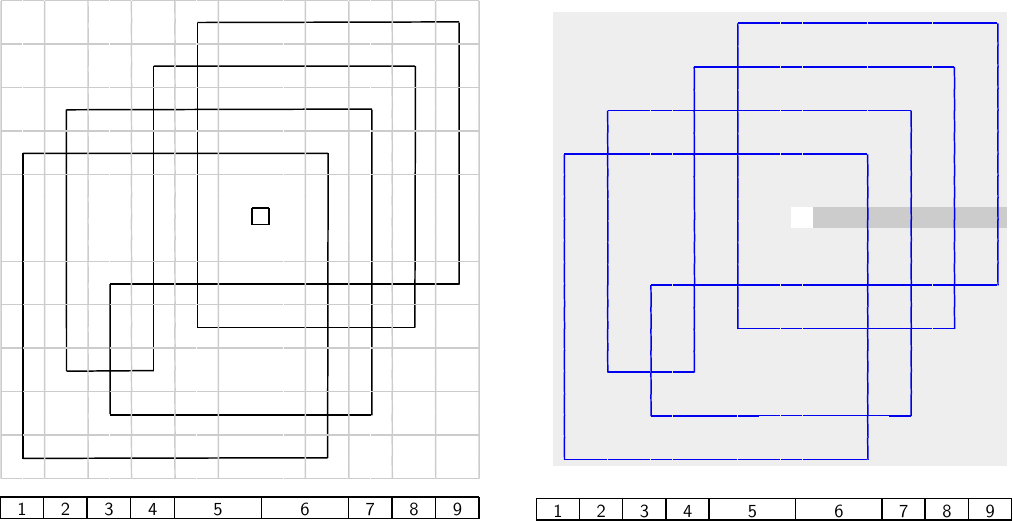}};
    \begin{scope}[x={(image.south east)},y={(image.north west)}]
        \node [font=\scriptsize] at (0.0, 0.62) {$\widetilde{C_{01}'}$};
        \node [font=\scriptsize] at (0.048, 0.62) {$\widetilde{C_{02}'}$};
        \node [font=\scriptsize] at (0.535, 0.62) {$B_{01}$};
        \node [font=\scriptsize] at (0.58, 0.62) {$B_{02}$};
        \node [below] at (0.25, -0.04) {(a)};
        \node [below] at (0.75, -0.04) {(b)};
    \end{scope}
\end{tikzpicture}
\caption{(a) The link $\widetilde{C_{0k}'}$ in grid position ($k=1, 2$). (b) The guide line $B_{0k}$ drawn on $S^1 \times D^1$ ($k=1, 2$).}
\label{fig:X0208}
\end{figure}

We draw the knots $\widetilde{C_{0k}'}$ ($k=1, 2$) in grid position on the 0-handle $S^1 \times D^1$, which is represented as a gray square, to serve as the guide lines $B_{0k}$ (see Figure~\ref{fig:X0208}(b)). 
A small square represents the hole of $S^1 \times D^1$, and the strands of the link cannot be moved across this small square by an isotopy.

From Column~1 to Column~8, when a vertical segment of a guide line $B_{0k}$ has an NW corner (resp.\ an NE corner), we attach a 1-handle to the boundary of the 0-handle to lift this segment, as shown in Figure~\ref{fig:X0205}(a) (resp.\ (b)). 
A newly added simple closed curve $C_i$ is formed, which does not cross the hole of $S^1 \times D^1$.

Then, by removing the guide lines $B_{0k}$, we construct the PALF $SF$ whose total space is diffeomorphic to $S^1 \times D^3$ (see Figure~\ref{fig:X0209}(a)).
Furthermore, by placing the closed curves $C_{0k}$ exactly where the guide lines $B_{0k}$ were removed, we obtain the PALF $P$ (see Figure~\ref{fig:X0209}(b)).
We assign the monodromy factorization $(C_8, C_7, C_6, C_5, C_4, C_3, C_2, C_1)$ to the PALF $SF$, and the monodromy factorization $(C_{01}, C_{02}, C_8, C_7, C_6, C_5, C_4, C_3, C_2, C_1)$ to the PALF $P$.

\begin{figure}[htbp]
\centering  
\begin{tikzpicture}
    \node[anchor=south west, inner sep=0] (image) at (0,0)  {\includegraphics[scale=0.75]{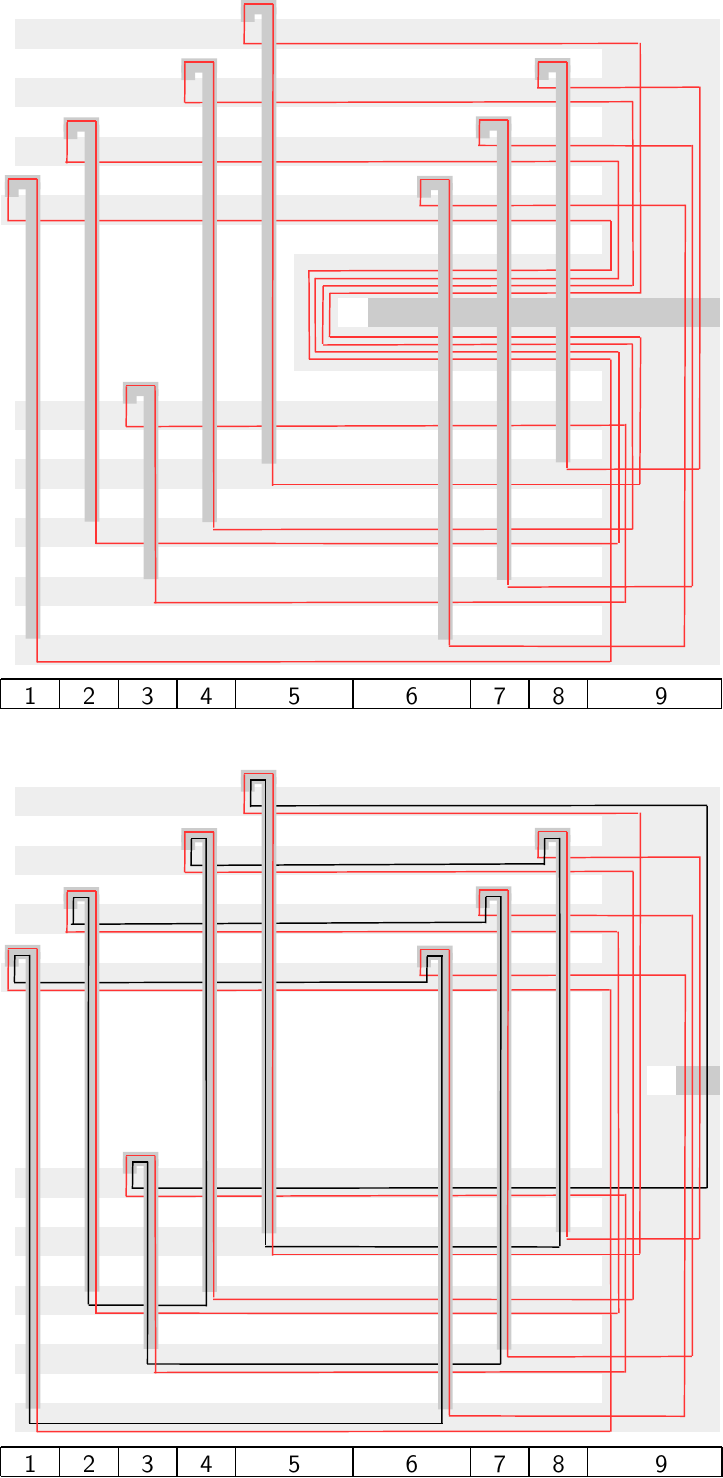}};
    \begin{scope}[x={(image.south east)},y={(image.north west)}]
        \node [font=\scriptsize] at (0.01, 0.3) {$C_{01}$};
        \node [font=\scriptsize] at (0.09, 0.3) {$C_{02}$};
        \node at (0.5, 0.495) {(a)};
        \node [below] at (0.5, -0.01) {(b)};
    \end{scope}
\end{tikzpicture}
\caption{(a) The PALF $SF$ with the monodromy factorization $(C_8, C_7, C_6, C_5, C_4, C_3, C_2, C_1)$.  
(b) The PALF $P$ with the monodromy factorization $(C_{01}, C_{02}, C_8, C_7, C_6, C_5, C_4, C_3, C_2, C_1)$. $C_i$ ($1 \leq i \leq 8$) denotes a red simple closed curve passing over the 1-handle in the $i$-th column.}
\label{fig:X0209}
\end{figure}

We can verify that $SF$ and $P$ are indeed PALFs by checking the conditions given in Subsection~2.3 of \cite{Tanaka2025construction}.
By the same argument as in Subsection~\ref{subsec:construction1}, we can show that the PALF $P$ is diffeomorphic to the original Stein surface $\Pi$.
As a corollary to Theorem~3.2 in \cite{Tanaka2025construction}, we obtain the following.

\begin{cor}\label{cor:construction2}
The total space of the PALF corresponding to the ordered collection of simple closed curves on the surface with boundary, obtained by the modified construction above, is diffeomorphic to the handlebody consisting of a 0-handle, $\ell$ 1-handles, and $m$ 2-handles attached along the Legendrian link components $\widetilde{C_{0k}}$ ($1 \leq k \leq m$) with framings $tb(\widetilde{C_{0k}})-1$.
\end{cor}

\subsection{Remarks on the construction method}\label{subsec:comment}

To apply our construction method in practice, we explain the allowed and prohibited operations using concrete examples. For this purpose, it is crucial to distinguish between the intrinsic information of the PALF constructed by our method and the induced auxiliary information. 

As concrete examples of the constructed PALFs, we consider the regular fibers shown in Figures~\ref{fig:X0203} and \ref{fig:X0204}(d) in Subsection~\ref{subsec:construction1}. The regular fiber consists of one 0-handle and four 1-handles, where the two attaching spheres of each 1-handle are attached to the boundary of the 0-handle.

\vspace{0.5em}
\noindent\textbf{(1) Lifting the guide line $B_0$ over a 1-handle}

Lifting a vertical segment of the guide line $B_0$ over a 1-handle is required if and only if at least one of the following conditions is satisfied:
\begin{enumerate}
    \item[(a)] The vertical segment of $B_0$ has an NW corner. Since this NW corner corresponds to a left cusp of the original Legendrian knot, it is necessary to adjust (specifically, decrease by one) the writhe of $\overline{C_0}$ in the Kirby diagram corresponding to the constructed PALF. For details, we refer the reader to the proof of Lemma~3.1 in \cite{Tanaka2025construction}.
    \item[(b)] The vertical segment of $B_0$ intersects a horizontal segment. Since self-intersections of the guide line are not permitted on the 0-handle, the vertical segment must be lifted over a 1-handle.
    \item[(c)] The vertical segment of $B_0$ must be lifted over a 1-handle to allow the 0-handle to be deformed via an isotopy passing under this 1-handle, so that the boundary of the 0-handle can reach the required position to attach a 1-handle corresponding to a vertical segment in a column to the right.
\end{enumerate}
If none of (a), (b), or (c) applies, the vertical segment of the guide line $B_0$ is not (and does not need to be) lifted over a 1-handle. In particular, the vertical segment in the rightmost column (i.e., the column with the maximum index) never satisfies any of (a), (b), or (c); thus, it is never lifted over a 1-handle (for example, see the 5th column in Figures~\ref{fig:X0203} and \ref{fig:X0204}(d)).

\vspace{0.5em}
\noindent\textbf{(2) Deformation of the 0-handle via isotopy}
\begin{enumerate}
    \item[(a)] When lifting a vertical segment of the guide line $B_0$ over a 1-handle and deforming the 0-handle via an isotopy passing under this 1-handle, the boundary of the 0-handle must be pushed from left to right under the 1-handle. This requirement ensures that the curve $C_i$ associated with the canceling pair acquires an algebraic self-intersection, which is necessary for the corresponding knot to have a writhe of $-1$ in the Kirby diagram.
    \item[(b)] Provided that condition (a) is satisfied, the 0-handle can be freely deformed via isotopy. For example, the 0-handle may be deformed into a narrow neighborhood of $B_0$ (essentially retracting it to that neighborhood), or the deformation may be kept to a bare minimum. Since the curves $C_i$ ($1 \leq i \leq 4$) as well as $B_0$ (i.e., $C_0$) lie on the 0-handle, it is permissible to deform the 0-handle via isotopy into a shape that facilitates the placement of these curves (e.g., Figures~\ref{fig:X0302}(b), (d), (f) and \ref{fig:X0303}(b)).
\end{enumerate}

\vspace{0.5em}
\noindent\textbf{(3) Drawing the curves $C_i$ ($1 \leq i \leq 4$) arising from the canceling pairs}

Regarding Figure~\ref{fig:X0203}, the boundary of the 0-handle is a single circle. If we represent the attaching spheres of each 1-handle by its column number, reading counterclockwise from the top-right point yields the sequence of attaching spheres for the 1-handles in Column 2, Column 3, Column 1, \dots, and finally Column 1. Up to cyclic permutation, this sequence is denoted by $(2, 3, 1, 4, 2, 3, 4, 1)$. Since the curves $C_0, C_1, C_2, C_3, C_4$ are arranged from bottom to top on the 0-handle in the order $C_0, C_4, C_3, C_2, C_1$, these curves are uniquely placed up to isotopy. In particular, the algebraic intersection number between any two curves among $C_1, C_2, C_3, C_4$ on the 0-handle is either $0$ or $\pm 1$ (depending on the choice of orientations). If the arrangement of their attaching spheres along the boundary of the 0-handle is alternating (e.g., for $C_1$ and $C_2$), the algebraic intersection number is $\pm 1$; if it is not alternating (e.g., for $C_1$ and $C_4$), the algebraic intersection number is $0$. As part of the construction, the curves $C_1, C_2, C_3, C_4$ can be drawn freely on the 0-handle. (Alternatively, they may be drawn after the four 1-handles have been attached to the 0-handle and the overall shape of the regular fiber has been completely determined.)

The three items (1), (2), and (3) above constitute the essential rules for our construction method. We conclude this subsection by providing a concrete example and stating a general principle concerning the vertical translation of a knot in grid position.

Figure~\ref{fig:X0211}(a) depicts the figure-eight knot. The regular fiber of the PALF constructed from it is shown in Figure~\ref{fig:X0211}(b). In this figure, the boundary of the 0-handle passes under the 1-handle that lifts the vertical segment in Column~4. It is then pushed upward via an isotopy in Column~6 to reach the vicinity of the upper endpoint of the vertical segment in Column~5. One can readily verify that this regular fiber fully satisfies the conditions (1), (2), and (3) outlined above.

On the other hand, Figure~\ref{fig:X0211}(c) illustrates the regular fiber of a PALF constructed from the knot in grid position obtained by vertically translating the 1st row of the figure-eight knot in Figure~\ref{fig:X0211}(a) to the 6th row. Via an isotopy passing through the intermediate state shown in Figure~\ref{fig:X0211}(d), the regular fiber in Figure~\ref{fig:X0211}(c) can be deformed into the one in Figure~\ref{fig:X0211}(b). Furthermore, it can be verified that the sequence of the attaching spheres of the 1-handles along the boundary of the 0-handle is identical up to cyclic permutation in both Figures~\ref{fig:X0211}(b) and (c).

In general, the following principle holds. Suppose that a Stein 2-handlebody is given as the knot trace of a knot in grid position. Applying a vertical translation to this knot yields another knot in grid position, defining a new knot trace. Let $F$ be the regular fiber of the PALF $P$ constructed from the original knot trace using our method. Then, by applying our method to the knot trace of the translated knot, we can obtain another PALF $P'$ whose regular fiber $F'$ is isotopic to $F$.

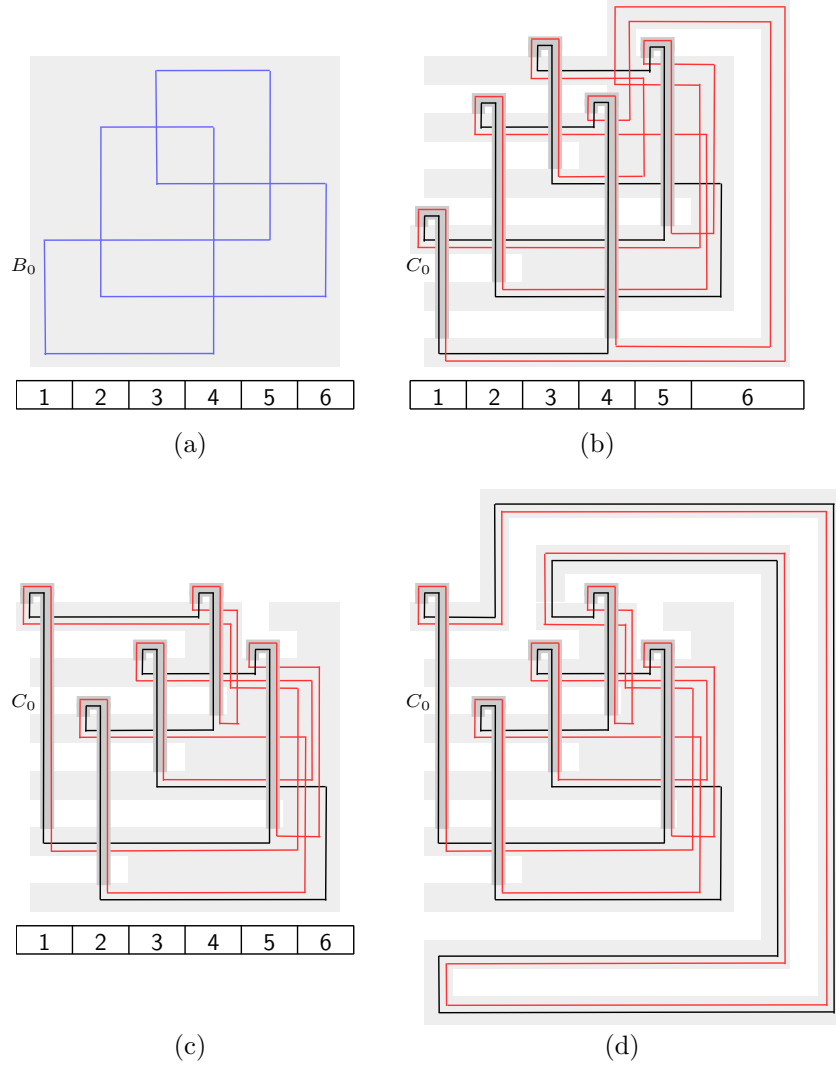
\begin{figure}[htbp]
\centering  
\begin{tikzpicture}
    \node[anchor=south west, inner sep=0] (image) at (0,0)  {\includegraphics[scale=0.75]{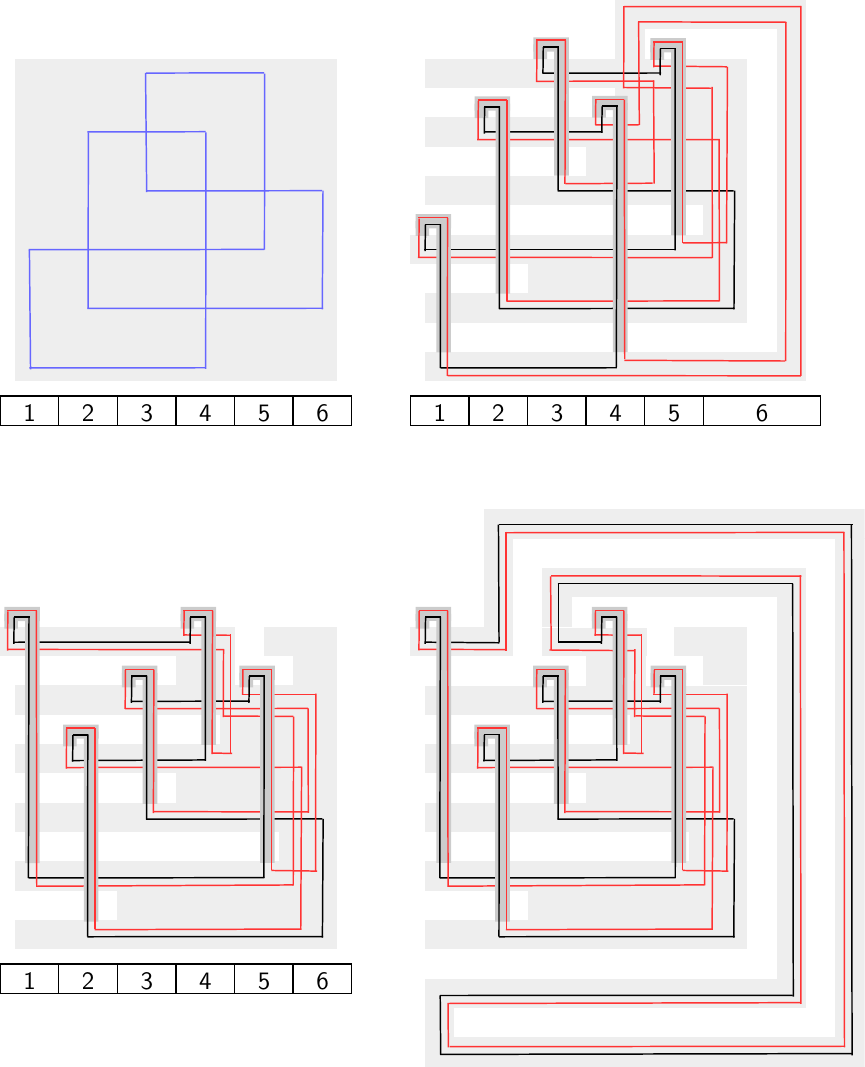}};
    \begin{scope}[x={(image.south east)},y={(image.north west)}]
        \node [font=\scriptsize] at (0.01, 0.74) {$B_{0}$};
        \node [font=\scriptsize] at (0.01, 0.315) {$C_{0}$};
        \node [font=\scriptsize] at (0.485, 0.74) {$C_{0}$};
        \node [font=\scriptsize] at (0.485, 0.315) {$C_{0}$};
        \node at (0.21, 0.565) {(a)};
        \node at (0.7, 0.565) {(b)};
        \node [below] at (0.21, 0.0) {(c)};
        \node [below] at (0.73, 0.0) {(d)};
    \end{scope}
\end{tikzpicture}
\caption{(a) The guide line $B_0$ drawn on the 0-handle $D^2$. (b) and (c) The PALFs with the monodromy factorization $(C_0, C_5, C_4, C_3, C_2, C_1)$. $C_i$ ($1 \leq i \leq 5$) denotes a red simple closed curve passing over the 1-handle in the $i$-th column. (d) An intermediate state during the isotopy.}
\label{fig:X0211}
\end{figure}

\section{Applications to knot traces} \label{sec:application}

In this section, we present two applications of our construction method.

\subsection{Twist knots} \label{subsec:twist}

As an application, we prove as a theorem the existence of a PALF whose regular fiber has genus $1$ and whose total space is diffeomorphic to the knot trace of a Legendrian positive twist knot.
Furthermore, we show that if the deformation of the boundary of the 0-handle is restricted to horizontal isotopies, the genus becomes larger.

The positive twist knot $W_s$ is the knot shown in Figure~\ref{fig:X0301}, where we assume $s \geq 1$.

\begin{figure}[htbp]
\centering  
\begin{tikzpicture}
    \node[anchor=south west, inner sep=0] (image) at (0,0)  {\includegraphics[scale=0.85]{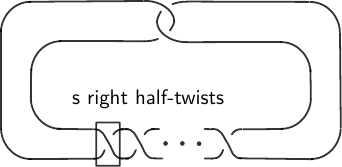}};
    \begin{scope}[x={(image.south east)},y={(image.north west)}]
    \end{scope}
\end{tikzpicture}
\caption{The positive twist knot $W_s$ ($s \geq 1$).}
\label{fig:X0301}
\end{figure}

\begin{thm}\label{thm:twist}
Suppose that a Stein surface is given as a 2-handlebody consisting of a single 0-handle and a single 2-handle attached along a Legendrian knot. If this Legendrian knot is the positive twist knot $W_s$ ($s \geq 1$), then there exists a PALF whose regular fiber has genus $1$ and whose total space is diffeomorphic to the Stein surface.
\end{thm}

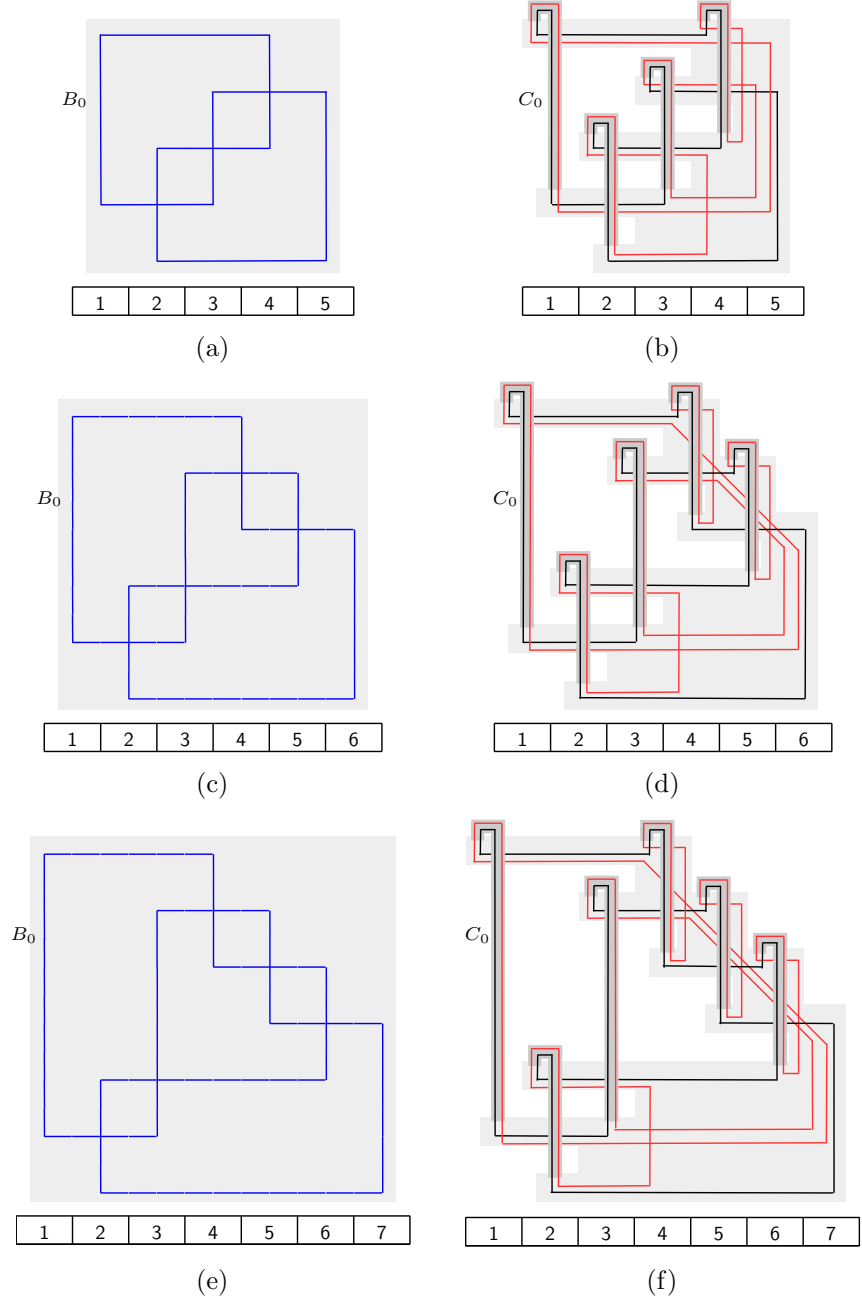
\begin{figure}[htbp]
\centering  
\begin{tikzpicture}
    \node[anchor=south west, inner sep=0] (image) at (0,0)  {\includegraphics[scale=0.75]{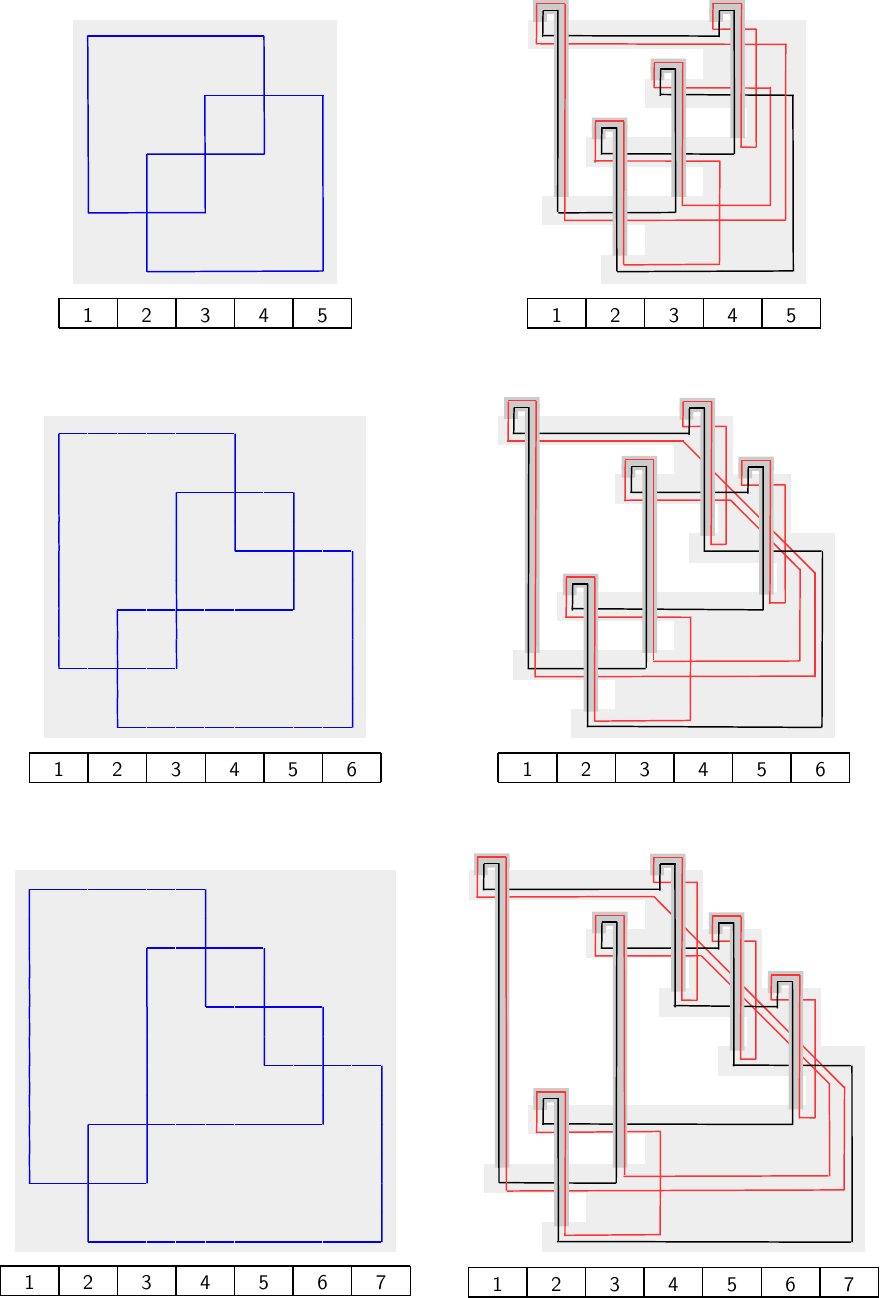}};
    \begin{scope}[x={(image.south east)},y={(image.north west)}]
        \node [font=\scriptsize] at (0.01, 0.25) {$B_0$};
        \node [font=\scriptsize] at (0.547, 0.25) {$C_0$};
        \node [font=\scriptsize] at (0.04, 0.6) {$B_0$};
        \node [font=\scriptsize] at (0.58, 0.6) {$C_0$};
        \node [font=\scriptsize] at (0.07, 0.92) {$B_0$};
        \node [font=\scriptsize] at (0.61, 0.92) {$C_0$};

        \node [below] at (0.233, -0.01) {(e)};
        \node [below] at (0.765, -0.01) {(f)};
        \node  at (0.233, 0.37) {(c)};
        \node  at (0.765, 0.37) {(d)};
        \node  at (0.233, 0.72) {(a)};
        \node  at (0.765, 0.72) {(b)};
    \end{scope}
\end{tikzpicture}
\caption{Regular fibers constructed from the knot traces along the positive twist knots $W_1$, $W_2$, and $W_3$. Monodromy factorizations: (b) $(C_0, C_4, C_3, C_2, C_1)$, (d) $(C_0, C_5, C_4, C_3, C_2, C_1)$, and (f) $(C_0, C_6, C_5, C_4, C_3, C_2, C_1)$. $C_i$ ($1 \leq i \leq 6$) denotes a red simple closed curve passing over the 1-handle in the $i$-th column.}
\label{fig:X0302}
\end{figure}

\begin{proof}
We use the construction method of PALFs presented in Subsection~\ref{subsec:construction1}. 
The cases where the Legendrian knots are the positive twist knots $W_1$, $W_2$, and $W_3$ are shown in Figures~\ref{fig:X0302}(a) and (b), (c) and (d), and (e) and (f), respectively. 
Figures~\ref{fig:X0302}(a), (c), and (e) show the Legendrian knot converted into a knot $\widetilde{C_0}$ in grid position, a copy of which is drawn on the 0-handle $D^2$ to serve as the guide line $B_0$. 
Figures~\ref{fig:X0302}(b), (d), and (f) show the regular fibers and the vanishing cycles of the constructed PALF $P$ in each case, respectively.

For $W_1, W_2$, and $W_3$, the numbers of boundary components are $3, 4$, and $5$, while the numbers of 1-handles are $4, 5$, and $6$, respectively. 
From these constructions, for $W_s$ ($s \geq 1$), the number of boundary components of the regular fiber is $s + 2$, and the number of 1-handles is $s + 3$. 
Using the Euler characteristic formula
\[
2 - 2g - (\text{the number of boundary components}) = 1 - (\text{the number of 1-handles}),
\]
we obtain the genus $g = 1$.
\end{proof}

In Figures~\ref{fig:X0302}(d) and (f), the deformation of the boundary of the 0-handle is not purely horizontal. If we deform the 0-handle into a comb-like domain by an isotopy parallel to each row, then for $W_2$, the regular fiber has $2$ boundary components, $5$ 1-handles, and genus $2$; and for $W_3$, it has $1$ boundary component, $6$ 1-handles, and genus $3$.

\subsection{Torus knots} \label{subsec:torus}

As a preparation for the next section, and as an application, we present the following theorem regarding the knot trace of a Legendrian positive torus knot $T_{2, 2n+1}$.

\begin{thm}\label{thm:torus}
Suppose that a Stein surface is given as a 2-handlebody consisting of a single 0-handle and a single 2-handle attached along a Legendrian knot.
If the Legendrian knot is the positive torus knot $T_{2, 2n+1}$, then there exists a PALF whose total space is diffeomorphic to the Stein surface, and whose regular fiber has genus $1$ and $2n+1$ boundary components.
\end{thm}

\begin{figure}[htbp]
\centering  
\begin{tikzpicture}
    \node[anchor=south west, inner sep=0] (image) at (0,0)  {\includegraphics[scale=0.7]{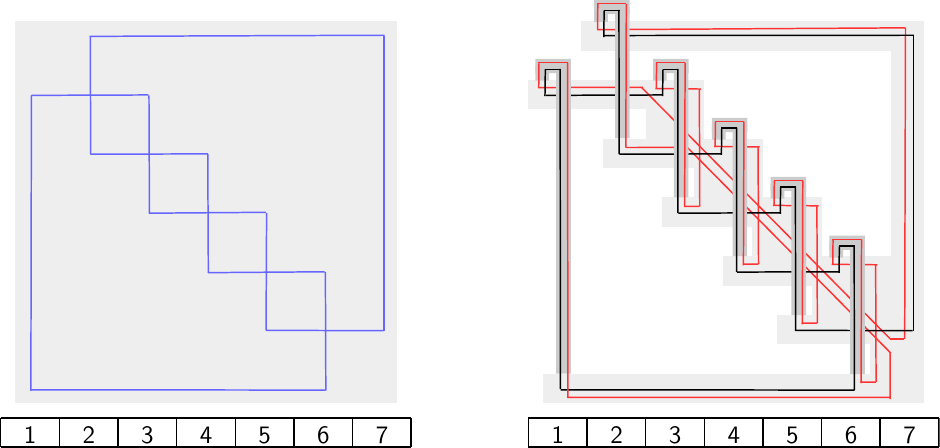}};
    \begin{scope}[x={(image.south east)},y={(image.north west)}]
        \node [font=\scriptsize] at (0, 0.572) {$B_0$};
        \node [font=\scriptsize] at (0.568, 0.572) {$C_0$};

        \node [below] at (0.224, -0.03) {(a)};
        \node [below] at (0.783, -0.03) {(b)};

    \end{scope}
\end{tikzpicture}
\caption{The regular fiber constructed from the knot trace along the positive torus knot $T_{2,5}$. Monodromy factorization: $(C_0, C_6, C_5, C_4, C_3, C_2, C_1)$. $C_i$ ($1 \leq i \leq 6$) denotes a red simple closed curve passing over the 1-handle in the $i$-th column.}
\label{fig:X0303}
\end{figure}

\begin{proof}
The case where the Legendrian knot is the positive torus knot $T_{2,3}$ is shown in Figures~\ref{fig:X0204}(a) and (d). 
The case where the Legendrian knot is the positive torus knot $T_{2,5}$ is shown in Figures~\ref{fig:X0303}(a) and (b). 
Figure~\ref{fig:X0303}(a) shows the Legendrian knot converted to the grid-position knot $\widetilde{C_0}$, a copy of which is drawn on the 0-handle $D^2$ to serve as the guide line $B_0$. 
Figure~\ref{fig:X0303}(b) shows the regular fiber and the vanishing cycles of the PALF $P$ constructed via the method in Subsection~\ref{subsec:construction1}.

For $T_{2,3}$ and $T_{2,5}$, the numbers of boundary components are $3$ and $5$, while the numbers of 1-handles are $4$ and $6$, respectively. 
From these constructions, for $T_{2, 2n+1}$ ($n \geq 1$), the number of boundary components of the regular fiber is $2n+1$, and the number of 1-handles is $2n+2$. 
Using the Euler characteristic formula
\[
2 - 2g - (\text{the number of boundary components}) = 1 - (\text{the number of 1-handles}),
\]
we obtain the genus $g = 1$.
\end{proof}

\section{Comparison with open book decompositions} \label{sec:open_book}

It is well known that the contact structure induced on the boundary of a Stein surface is supported by the open book decomposition induced by the PALF that the Stein surface admits (cf.\ \cite[Subsection~2.1]{MR3609904}). When an open book is induced by a PALF, its page corresponds to the regular fiber of the PALF, and its monodromy can be identified with that of the PALF.

In Subsection~2.1, we constructed a PALF $P$ whose total space is diffeomorphic to the knot trace of the positive torus knot $T_{2,3}$ (see Figure~\ref{fig:X0204}(d)). Our construction process involves a PALF $SF$ whose total space is diffeomorphic to $D^4$ (see Figure~\ref{fig:X0204}(c)). The PALF $P$ is obtained by embedding a knot $\widetilde{C'_0}$, derived from the Legendrian knot $T_{2,3}$, as a vanishing cycle into the regular fiber of $SF$.

On the other hand, Avdek proved that a Legendrian torus knot $K = T_{2, 2n+1}$ in $(S^{3},\xi_{std})$ with Thurston--Bennequin number $tb(K)=2n-1$ is contained in the page $\Sigma$ of an open book $(\Sigma, \Phi)$ supporting $(S^{3},\xi_{std})$ \cite[Theorem~5.2]{MR3071137}. In this construction, the page $\Sigma$ has the topological type of a $(2n+1)$-punctured torus, which corresponds to a surface of genus $1$ with $2n+1$ boundary components (Figure~\ref{fig:X0401}). 

Note that this topological type coincides not only with that of the regular fiber of the PALF $SF$ (and consequently $P$) described in Subsection~2.1 for the case of $T_{2,3}$, but also with the general result for $T_{2, 2n+1}$ presented in Subsection~\ref{subsec:torus}.

\begin{figure}[htbp]
\centering  
\begin{tikzpicture}
    \node[anchor=south west, inner sep=0] (image) at (0,0)  {\includegraphics[scale=0.52]{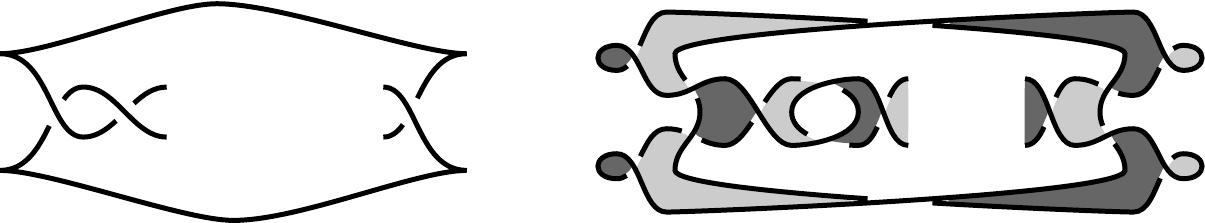}};
    \begin{scope}[x={(image.south east)},y={(image.north west)}]
        \node  at (0.19, -0.19) {$(2n+1)$ positive crossings};
        \node [below] at (0.20, -0.35) {(a)};
        \node [below] at (0.75, -0.35) {(b)};
    \end{scope}
\end{tikzpicture}
\caption{(a) A Legendrian torus knot $T_{2, 2n+1}$. (b) A page of an open book supporting $(S^3, \xi_{std})$ and containing $T_{2, 2n+1}$, constructed using Avdek's Algorithm~2.}
\label{fig:X0401}
\end{figure}
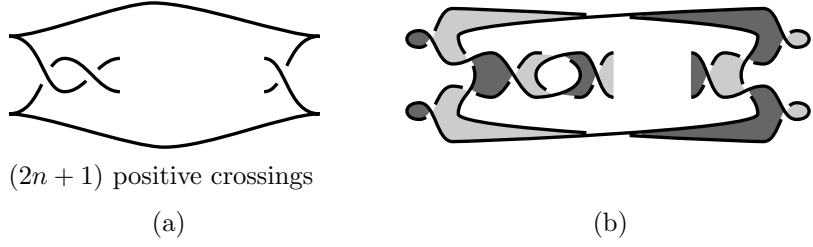

First, we consider the case $n=1$, i.e., the torus knot $T_{2, 3}$ with $tb(K)=1$. In the following, we compare our PALF $SF$ with the open book decomposition $OB$ associated with $T_{2, 3}$.

\vspace{0.5em}
\noindent
\textbf{(1) The page and monodromy of the open book $OB$}

Details for the torus knot $T_{2, 3}$ are provided in \cite[Figure~1, Subsection 1.3]{MR3071137}. We reproduce this figure, with additional colored curves, in Figures~\ref{fig:X0402}(e) and (f). We refer to this open book as $OB$. The page has genus $1$ and $3$ boundary components, which are colored red, blue, and orange. The letters $a$, $b$, $c$, and $d$ correspond to simple closed curves $\gamma_{a}$, $\gamma_{b}$, $\gamma_{c}$, and $\gamma_{d}$ on the surface, which we indicate in magenta, dark green, blue, and red, respectively. The monodromy of the associated open book supporting $(S^{3},\xi_{std})$ is given by $\Phi = D^{+}_{\gamma_{d}} \circ D^{+}_{\gamma_{c}} \circ D^{+}_{\gamma_{b}} \circ D^{+}_{\gamma_{a}}$, where the product is read from right to left (standard functional composition).

\begin{figure}[htbp]
\centering  
\begin{tikzpicture}
    \node[anchor=south west, inner sep=0] (image) at (0,0)  {\includegraphics[scale=0.75]{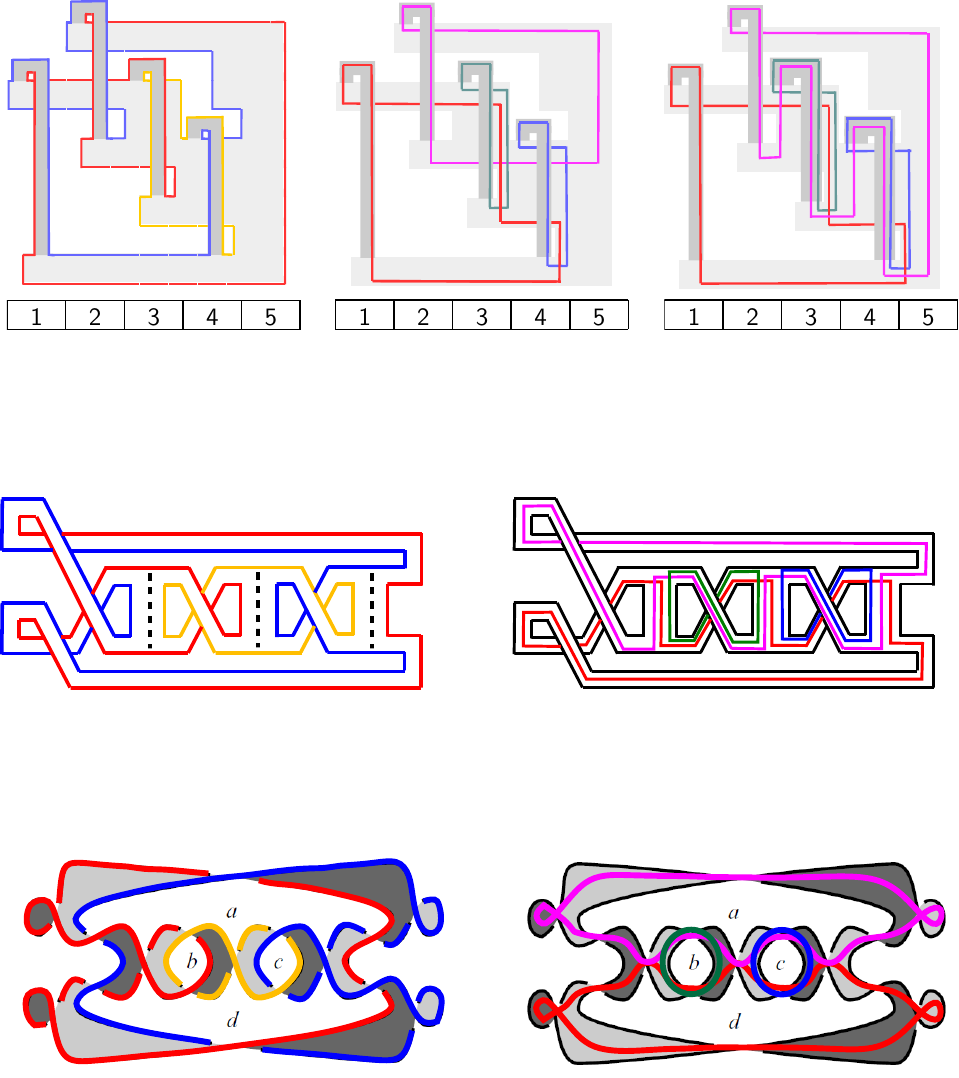}};
    \begin{scope}[x={(image.south east)},y={(image.north west)}]
        \node [below] at (0.233, -0.01) {(f)};
        \node [below] at (0.765, -0.01) {(g)};
        \node  at (0.233, 0.325) {(d)};
        \node  at (0.765, 0.325) {(e)};
        \node  at (0.16, 0.66) {(a)};
        \node  at (0.505, 0.66) {(b)};
        \node  at (0.85, 0.66) {(c)};
    \end{scope}
\end{tikzpicture}
\caption{Comparison between the PALFs ($SF$ and $SF'$) and the open book $OB$ for the torus knot $T_{2,3}$. (a) The regular fiber of the PALF $SF$. (b) The vanishing cycles of $SF$. (c) The vanishing cycles of the PALF $SF'$ obtained by elementary transformations. (d) and (e) The results of rotating (a) and (c) counterclockwise by $45^\circ$, respectively. (f) and (g) The page and monodromy of the open book $OB$. (d) and (e) coincide with (f) and (g), respectively, after performing half-twists along the three black dashed lines depicted in (d).}
\label{fig:X0402}
\end{figure}
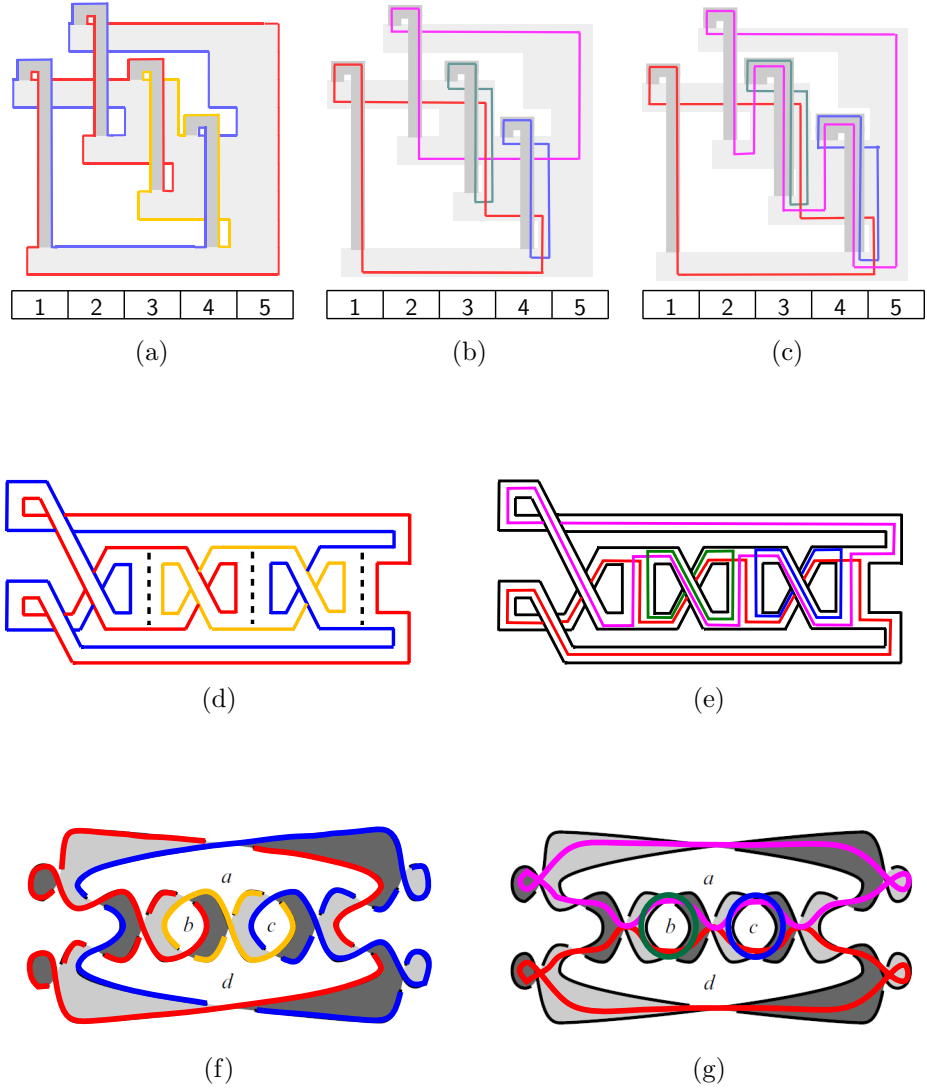

\vspace{0.5em}
\noindent
\textbf{(2) The regular fiber and monodromy of the PALF $P$}

Figure~\ref{fig:X0204}(c) shows the regular fiber and monodromy of the PALF $SF$ whose total space is diffeomorphic to $D^4$. (By embedding the vanishing cycle associated with the torus knot $T_{2, 3}$, we obtain the resulting PALF $P$.) The regular fiber of $SF$ has genus $1$ and $3$ boundary components, which we color red, blue, and orange as shown in Figure~\ref{fig:X0402}(a). Rotating Figure~\ref{fig:X0402}(a) counterclockwise by $45^\circ$ yields Figure~\ref{fig:X0402}(d). By performing a half-twist---bringing the lower strand upward through the front and the upper strand downward through the back---on the left side of each of the three black dashed lines, we can see that the configuration coincides with Figure~\ref{fig:X0402}(f). 

The monodromy of $SF$ is represented by a monodromy factorization $(C_4, C_3, C_2, C_1)$ consisting of four vanishing cycles. We indicate these curves in blue, dark green, magenta, and red, respectively (see Figure~\ref{fig:X0402}(b)). Applying elementary transformations twice to $C_2$ yields $C_2' = t^{-1}_{C_4}(t^{-1}_{C_3}(C_2))$, thereby modifying the monodromy factorization to $(C_2', C_4, C_3, C_1)$. Drawing the new vanishing cycle $C_2'$ in the same magenta color as $C_2$ yields Figure~\ref{fig:X0402}(c). Furthermore, since $C_4$ and $C_3$ are disjoint, they commute, and the monodromy factorization becomes $(C_2', C_3, C_4, C_1)$. We denote the PALF associated with this factorization by $SF'$. 

Rotating Figure~\ref{fig:X0402}(c) counterclockwise by $45^\circ$ yields Figure~\ref{fig:X0402}(e). As with Figure~\ref{fig:X0402}(d), performing half-twists at the three locations reveals that the configuration coincides with Figure~\ref{fig:X0402}(g). Since the product in the monodromy factorization $(C_2', C_3, C_4, C_1)$ is read from left to right, it perfectly coincides with the monodromy of the open book $OB$.

From the above, the regular fiber and monodromy of the PALF $SF'$, which is Lefschetz equivalent to the PALF $SF$, coincide with the page and monodromy of the open book $OB$, respectively. 

Next, we consider the general case for $n \geq 1$, namely the positive torus knot $T_{2, 2n+1}$. By Theorem~\ref{thm:torus}, the regular fiber of the constructed PALF has genus $1$ and $2n+1$ boundary components, which exactly matches the topological type of the page in the corresponding open book. Furthermore, by applying the same geometric deformations to the regular fiber and elementary transformations to the monodromy as we did for the $n=1$ case, we can verify that the regular fiber and monodromy of the PALF coincide with the page and monodromy of the associated open book, respectively.

\section{Variations of the construction method} \label{sec:variation}

Exploring variations of our construction method provides valuable insights into the method itself. While our original method lifts the vertical segments of the guide line $B_0$ on the 0-handle $D^2$ over 1-handles starting from the 1st column, there is an alternative method that performs this operation starting from the rightmost column (i.e., the column with the maximum index).

First, recall that we constructed the PALF shown in Figure~\ref{fig:X0203} from the grid-position knot $\widetilde{C_0'}$ shown in Figure~\ref{fig:X0201}(a). Figure~\ref{fig:X0501}(a) reproduces Figure~\ref{fig:X0203}. On the other hand, Figure~\ref{fig:X0501}(b) depicts a PALF constructed from the same guide line $B_0$ on the 0-handle $D^2$ derived from $\widetilde{C_0'}$ in Figure~\ref{fig:X0201}(a), but by performing the operations from the 5th column down to the 2nd column. Note that in the grid-position knot $\widetilde{C_0'}$ in Figure~\ref{fig:X0201}(a), the vertical segments in the 5th and 4th columns have SE (southeast) corners, which correspond to the right cusps of the original Legendrian knot.

Let $P$ be the PALF constructed by applying the original method, which starts operations from the 1st column, to the guide line $B_0$ placed on the 0-handle $D^2$. Let $\hat{P}$ be the PALF constructed by applying the alternative method, which starts operations from the column with the maximum index, to the guide line $\widehat{B_0}$ placed on the 0-handle $D^2$, where $\widehat{B_0}$ is obtained by rotating $B_0$ by $180^\circ$. It can be seen that we can construct a PALF $\hat{P}$ whose regular fiber perfectly coincides with the $180^\circ$ rotated regular fiber of $P$.

Furthermore, while our original method lifts the vertical segments of the guide line $B_0$ over 1-handles above the 0-handle $D^2$, there are also construction methods that push the horizontal segments downward so that they pass behind the 0-handle via 1-handles. Figure~\ref{fig:X0501}(c) shows a PALF constructed from the guide line $B_0$ associated with $\widetilde{C_0'}$ in Figure~\ref{fig:X0201}(a) by performing this push-down operation from the 5th row to the 2nd row. Since the 1-handles pass behind the gray 0-handle, which is deformed into a comb-like domain, portions of the vanishing cycles located behind the 0-handle (e.g., the part of the 1-handle in the 5th row passing behind the teeth of the 0-handle in Columns~2 to 4) are represented by dashed lines to indicate this depth. Figure~\ref{fig:X0501}(d) illustrates a PALF constructed similarly by performing the operations from the 1st row to the 4th row.

It can also be seen that if we apply this push-down construction method, which produces Figure~\ref{fig:X0501}(c) (resp., Figure~\ref{fig:X0501}(d)), to the guide line $B_0$ after applying an isotopy in $S^3$ that rotates it by $180^\circ$ around the downward-sloping diagonal line connecting the top-left and bottom-right corners (resp., the upward-sloping diagonal line connecting the bottom-left and top-right corners) of the 0-handle $D^2$, the regular fiber of the resulting PALF is equivalent to the surface obtained by applying the same isotopy in $S^3$ to the regular fiber of the PALF constructed by the original method.

Although combining these techniques to construct a single surface yields new surfaces, the resulting surface does not necessarily admit a PALF structure. Exploring the effective application of these combined techniques remains an interesting direction for future research.

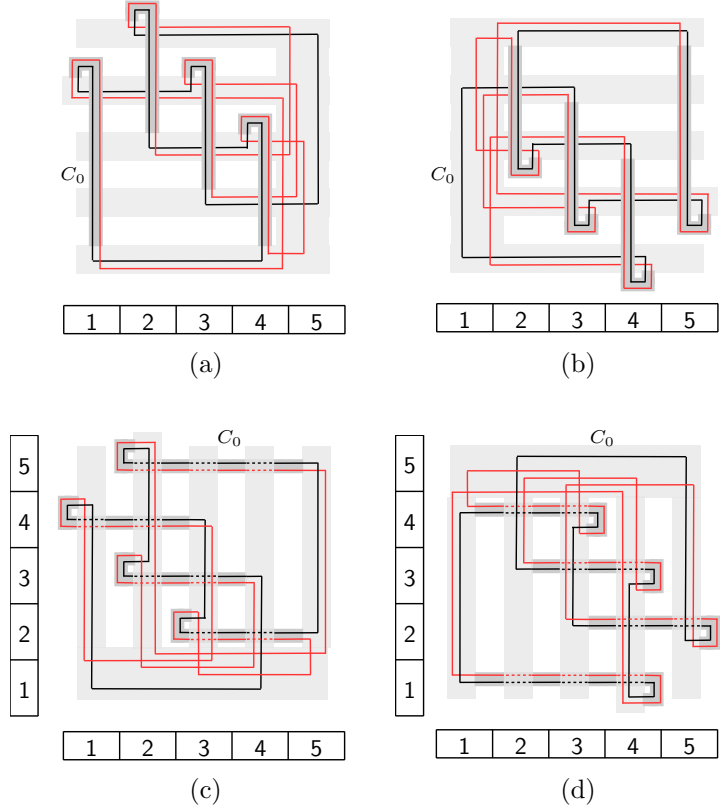
\begin{figure}[htbp]
\centering  
\begin{tikzpicture}
    \node[anchor=south west, inner sep=0] (image) at (0,0)  {\includegraphics[scale=0.75]{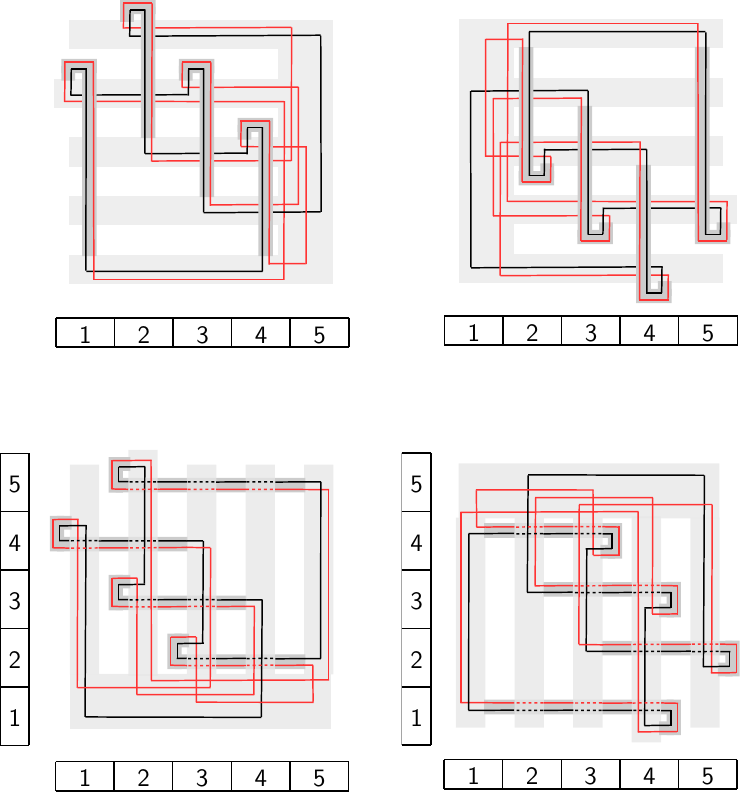}};
    \begin{scope}[x={(image.south east)},y={(image.north west)}]
        \node [font=\scriptsize] at (0.09, 0.77) {$C_{0}$};
        \node [font=\scriptsize] at (0.31, 0.424) {$C_{0}$};
        \node [font=\scriptsize] at (0.61, 0.77) {$C_{0}$};
        \node [font=\scriptsize] at (0.834, 0.424) {$C_{0}$};
        \node at (0.276, 0.52) {(a)};
        \node at (0.8, 0.52) {(b)};
        \node [below] at (0.276, -0.01) {(c)};
        \node [below] at (0.8, -0.01) {(d)};
    \end{scope}
\end{tikzpicture}
\caption{(a) The PALF $P$ with the monodromy factorization $(C_{0}, C_4, C_3, C_2, C_1)$. $C_i$ ($1 \leq i \leq 4$) denotes a red simple closed curve passing over the 1-handle in the $i$-th column. (b) The PALF with the monodromy factorization $(C_{0}, C_2, C_3, C_4, C_5)$. $C_i$ ($2 \leq i \leq 5$) denotes a red simple closed curve passing over the 1-handle in the $i$-th column. (c) The PALF with the monodromy factorization $(C_{0}, C_2, C_3, C_4, C_5)$. $C_i$ ($2 \leq i \leq 5$) denotes a red simple closed curve passing over the 1-handle in the $i$-th row. (d) The PALF with the monodromy factorization $(C_{0}, C_1, C_2, C_3, C_4)$. $C_i$ ($1 \leq i \leq 4$) denotes a red simple closed curve passing over the 1-handle in the $i$-th row. In (c) and (d), portions of the curves passing behind the 0-handle are represented by dashed lines. Note that the monodromy factorizations in (c) and (d) are described with respect to the regular fibers that have been turned over by a $180^\circ$ rotation via an isotopy in $S^3$.}
\label{fig:X0501}
\end{figure}


\bibliographystyle{amsalpha}
\bibliography{ALF_1}

\end{document}